\documentclass[11pt,reqno]{amsart}
\usepackage[T2A]{fontenc}
\usepackage[cp1251]{inputenc}
\usepackage[russian,english]{babel}

\usepackage{izvuz6}
\usepackage[pdftex,unicode,colorlinks,linkcolor=blue,
citecolor=red,bookmarksopen,pdfhighlight=/N]{hyperref}

\usepackage{indentfirst}
\usepackage{verbatim}
\theoremstyle{plain}

\newtheorem{theorem}{Теорема}

\usepackage{indentfirst}

\usepackage{amsfonts,amssymb,mathrsfs,amscd}
\usepackage{amsmath,latexsym}
\usepackage{amsthm}
\usepackage{verbatim}
\usepackage{eucal}
\usepackage{enumerate}
\newtheorem*{theoremF}{Теорема F}%%\def\thetheoremF{\hspace{-2mm}}
\newtheorem*{lemmao}{Лемма}%%\def\thelemmao{\hspace{-1mm}}%%[section]
\theoremstyle{definition}
\theoremstyle{definition}

\newtheorem{remark}{Замечание}
\newtheorem{example}{Пример}

\renewcommand{\Re}{{\rm Re \,}}
\renewcommand{\leq}{\leqslant}
\renewcommand{\geq}{\geqslant}
\renewcommand{\d}{{\rm \,d}}
\def\Re{\operatorname{Re}}

\def\RR{\mathbb R}
\def\CC{\mathbb C}
\def\NN{\mathbb N}
\def\DD{\mathbb D}
\def\BB{\mathbb B}
\DeclareMathOperator{\Hol}{Hol}

\tit{\MakeUppercase{От интегральных оценок  функций к  равномерным. 
II. Точные версии}}
\shorttit{\MakeUppercase{От интегральных оценок  функций к  равномерным. II. Точные версии}} 
\author{\MakeUppercase{Р.\,А.~Баладай, Б.\,Н.\,Хабибуллин}}
\tauthor{\MakeUppercase{Р.\,А.~Баладай, Б.\,Н.\,Хабибуллин}}  %% верхний колонтитул на четных страницах

\god{201\_}         %% заполняется редакцией
\pp{\pageref{firstpage}--\pageref{lastpage}}

\begin{document}
\selectlanguage{russian}
\maketit
\label{firstpage}%% Не удалять!!!

%% Аннотация на русском
\abstract{В теории функций комплексных переменных нечасто встречаются точные поточечные оценки функций при известных интегральных ограничениях на их рост. Примером такого рода оценки может служит поточечная оценка модуля функции в пространстве Фока целых функций через интегральную норму этой функции.  Мы предлагаем некоторую функционально-аналитическую схему получения таких оценок в едином ключе и иллюстрируем ее на примерах классических пространств голоморфных  функций типа Фока\,--\,Баргмана и  Бергмана\,--\,Джрбашяна
в $n$-мерном комплексном пространстве, шаре, поликруге и т.\,п.} 

%% Ключевые слова на русском

\udk{517.55 + 517.987.1  + 517.576}  %% номер УДК статьи

%% Аннотация на английском
\englishabstract{Exact  pointwise estimates of the functions under certain integral constraints on their growth are not often met in the theory of functions of a complex variables. An example of this kind of estimation is the pointwise estimation of the module of function in the Fock space by integral norm of this function. We present some functional-analytic scheme for obtaining such estimates in a unified manner and illustrate it on the examples of classical  
Fock\,--\,Bargmann-type and Bergman\,--\,Djrbashian-type  spaces of holomorphic functions 
on n-dimensional complex spaces, balls, polidiscs etc.}

%% Ключевые слова на английском
\englishkeywords{image of measure, integral pre-norm, holomorphic function, automorphism, Fock space, Bergman space}  

\Received{\_.\_.201\_}  %% дата получения; заполняется редакцией

%% Поддержка и т.п.
\thnk{Работа выполнена при финансовой поддержке Российского фонда
фундаментальных исследований, проект №~16-01-00024-а.}
% можно использовать несколько раз
%\thnk{ }
%\thnk{ }

\section*{Введение} 
Как обычно, $\NN$,  $\RR$,  $\CC$  --- множества  всех {\it натуральных,\/} {\it  вещественных\/} и  {\it комплексных чисел\/} в стандартных интерпретациях; 
$\lambda$  ---  {\it мера Лебега на\/}  $\CC^n$, $n\in \NN$,    со стандартной {\it евклидовой нормой\/} 
$|z|:=\sqrt{|z_1|^2+\dots +|z_n|^2}$, \quad $z=(z_1, \dots, z_n)\in \CC^n$.
Для непустого ($\neq \varnothing$) открытого связного  множества, т.е. {\it области,\/}  $D \subset \CC^n$ 	 через  $\Hol (D)$ обозначаем   векторное пространство над $\CC$  всех голоморфных  функций в $D$. Для пары  $\alpha\in (0,+\infty)$ и $p\in (0,+\infty]$  
конечность следующих квазинорм
\begin{subequations}\label{ino}
\begin{align}
\|f\|_{p,\alpha}&:=\left(\frac{p\alpha}{2\pi}\int\limits_{\CC}\bigl|f(z)\bigr|^pe^{-\frac{p\alpha}{2}|z|^2}\d \lambda(z)\right)^{1/p} \quad \text{{\it при\/} $p\neq +\infty$}, 
\tag{\ref{ino}a}\label{{ino}a}
\\
 \|f\|_{+\infty,\alpha}&:= \sup_{z\in \CC}
\Bigl(\bigl|f(z)\bigr|e^{-\frac{\alpha}{2}|z|^2}\Bigr)
\tag{\ref{ino}b}\label{{ino}b}
\end{align}
\end{subequations}
%%где $\ess \sup$ --- существенная точная верхняя грань, 
для функций $f\in \Hol (\CC)$ определяет {\it пространства Фока\,--\,Баргмана\/} $F_{\alpha}^p(\CC)$ \cite{F}--\cite{Lind}.
\begin{theoremF}[{\cite[Теорема 2.7, Следствие 2.8]{F}}] 
Для любых $\alpha\in (0,+\infty)$,  $p\in (0,+\infty]$ и точки  $z\in \CC$ верна точная оценка $\bigl|f(z)\bigr|e^{-\frac{\alpha}{2}|z|^2}\leq \|f\|_{p,\alpha}$, а оценка $\|f\|_{+\infty,\alpha}\leq \|f\|_{p,\alpha}$ не улучшаема.
\end{theoremF}
Теорема F очень легко переносится на $\CC^n$ с $n>1$ и, скорее всего, где-либо приведена. Нам были доступны лишь оценки   без точных констант  \cite[2.1]{Mas}, \cite{Lind}. Во всяком случае, без претензий на новизну,  точная  оценка для пространства Фока\,--\,Баргмана --- это Пример \ref{ex1} ниже.  Для получения такого рода точных оценок  в едином ключе ниже в разделе \ref{sec:1}, Теорема \ref{th:1}, предлагается простая функционально-аналитическая конструкция. Эта конструкция применима не только к пространствам голоморфных  функций, как это делается в разделе \ref{sec:2}, Теорема \ref{th:2}. Но эта возможность применения Теоремы \ref{th:1} к <<неголоморфным>> ситуациям  для явно заданных пространств функций  здесь не рассматривается.  Примеры  \ref{ex1}--\ref{ex3} из раздела \ref{sec:3} 
иллюстрируют Теорему \ref{th:2} в случаях конкретных пространств голоморфных функций типа Фока\,--\,Баргмана и Бергмана\,--\,Джрбашяна  в $\CC^n$, в шаре и  в поликруге. В частности, они существенно дополняют и развивают  \cite[Пример 1]{BalKh}.

\section{Общая функционально-аналитическая  схема}\label{sec:1}  
Используются терминология и сведения из \cite{Sch}.
Пусть $X$ --- локально компактное пространство,  счетное в бесконечности, $\mathcal M^+(X)$ --- конус положительных борелевских мер или мер Радона на $X$. В этом разделе \ref{sec:1} всюду $\mu$ --- мера из $\mathcal M^+(X)$. Для  этой меры $\mu$ рассматриваем класс отображений $\mathfrak a\colon X  \to X$, удовлетворяющих условию  
 \begin{enumerate}
\item[(A)] {\it отображение $\mathfrak a$ 
$\mu$-измеримо  и  $\mu(\mathfrak a^{-1}K)<+\infty$ 
для любого компакта $K$ из $X$, где, как обычно, ${\mathfrak a}^{-1}K$ --- прообраз $K$.}
\end{enumerate}
Для таких отображений $\mathfrak a$ определен {\it образ\/ ${\mathfrak a}\mu \in \mathcal M^+(X)$ меры\/ $\mu$ при отображении\/} ${\mathfrak a}$, 
действующий по правилу 
$({\mathfrak a}\mu) (K)=\mu ({\mathfrak a}^{-1}K) \quad\text{на произвольных компактах $K$ из $X$.}$
В  частности, для любой функции  $f\colon X\to [-\infty, +\infty] $ справедливо равенство 
\begin{equation*}%%\label{key}
\int f \d ({\mathfrak a}\mu) =\int (f\circ {\mathfrak a}) \d \mu
\end{equation*} 
при условии {\it  $\mu$-интегрируемости суперпозиции\/}  $f\circ {\mathfrak a}$ (см. \cite[гл.~IV, \S~6, Теорема 60]{Sch}).

Зафиксируем пару точек  $x,y\in X$. Пусть отображение $\mathfrak a=\mathfrak a_x^y\colon X\to X$ удовлетворяет условию  (A) и 
\begin{equation}\label{c:a}
\mathfrak a_x^y (x)=y, \qquad \mathfrak a_x^y \mu=\mu; \qquad \text{$x,y\in X$ --- пара фиксированных точек}, 
\end{equation}
где второе  равенство означает, что мера $\mu$ {\it инвариантна относительно  отображения $\mathfrak a_x^y$.}

Пусть $U$ --- некоторый класс {\it полунепрерывных  сверху\/} функций $u\colon X\to [{-\infty},+\infty)$, {\it замкнутый относительно сложения\/} (поточечного)
 и  {\it инвариантный относительно отображения\/} $\mathfrak  a_x^y$ в том смысле, что $u\circ \mathfrak a_x^y\in U$ для любой функции $u\in U$. 

%%Ниже в этом п.~{\rm 1} по умолчанию  подынтегральные выражения с функциями из $U$ %%предполагаются %%интегрируемыми по соответствующей мере, т.\,е. интегралы от них  конечны.  В %%применениях к конкретным %%ситуациям  это требует непосредственной проверки.  
\begin{lemmao}\label{pr:1} 
	Пусть в  обозначениях и условиях этого раздела~{\rm \ref{sec:1}} функция $v\colon X\to \RR$ непрерывна,  
	\begin{equation}\label{c:va}
	(v-v\circ \mathfrak a_x^y)\in U,  \quad\text{а также, как следствие, 
		$u_x^y:= u\circ \mathfrak a_x^y+(v -v\circ \mathfrak a_x^y )\,{\in}\, U$} 
	\end{equation}
 для любой функции  $u\in U$. Тогда для каждой функции $u\in U$ и  каждой полунепрерывной сверху функции $\Phi \colon [-\infty,+\infty)\to \RR$   имеют место равенства
	\begin{equation}\label{0i0x}
		u(y)-v(y)+v(x){=}u_x^y(x),
\quad 		\int \Phi \circ (u-v) \, d \mu=\int \Phi \circ \bigl(u_x^y-v\bigr) \, d \mu 
\end{equation}
при условии, что существует конечный интеграл в левой части последнего равенства.
\end{lemmao}

\begin{proof} Первое равенство в \eqref{0i0x} следует из равенств
\begin{equation*}%%\label{key}
u_x^y(x) =u\bigl(\mathfrak a_x^y(x)\bigr)+v(x)-v\bigl(\mathfrak a_x^y(x)\bigr){=}u(y)+v(x)-v(y), 
\end{equation*}
использующих \eqref{c:a}. Также с помощью \eqref{c:a} из цепочки равенств
\begin{multline*}
\int \Phi \circ (u-v) \d \mu {=}
\int \Phi \circ (u-v) \d (\mathfrak a_x^y\mu)=
\int \Phi \circ (u\circ \mathfrak a_x^y -v\circ \mathfrak a_x^y) \d \mu\\
= \int \Phi \circ \bigl(\underbrace{u\circ \mathfrak a_x^y -v\circ \mathfrak a_x^y+v}_{u_x^y}-v\bigr) \d \mu
=\int \Phi \circ \bigl(u_x^y-v\bigr) \d \mu
\end{multline*}
получаем второе равенство в \eqref{0i0x}.
\end{proof}

Далее рассматриваем {\it непрерывные  положительные функции} 
\begin{equation}\label{Phi}
\Phi \colon \RR\to (0,+\infty), \quad q\colon \RR\to (0,+\infty), \quad Q\colon (0,+\infty) \to (0,+\infty),
\end{equation}
продолжимые по непрерывности в точки $\pm\infty,0$ значениями из $[0,+\infty]$. 
Введем  %%функционалы-
аналоги   <<{\it квазинорм\/}>> функции $u\in U$, а также, для каждой точки $x\in X$, функционала, действующего на функции ${{\mathsf u}}\in U$ по правилу  ${{\mathsf u}}\mapsto  (q\circ {{\mathsf u}})(x)\in (0,+\infty)$,
  а именно:
\begin{equation}\label{PhiQq}
\|u\|:=Q\biggl( \,\int \Phi\circ (u-v)\, d \mu\biggl)  \quad \text{на  $u\in U$}; \quad
\|\delta_x\|^*:=\sup_{\stackrel{0< \|\mathsf u\|< +\infty}{{{\mathsf u}}\in U}}\frac{q({{\mathsf u}}(x))}{%%Q(I_{\Phi}\bigr({{\mathsf u}})\bigl)
\|\mathsf u\|}>0  
\end{equation}
в предположении существования функции  $\mathsf u\in U$ с $\mathsf u(x)\neq -\infty$.
\begin{theorem}\label{th:1} В  условиях и обозначениях   \eqref{c:va}--\eqref{PhiQq} имеем
\begin{equation}\label{best}
q\bigl(u(y)-v(y)+v(x)\bigr) \leq \|\delta_x\|^* \cdot \|u\|
%% \cdot Q\bigl(I_{\Phi}(u)\bigr)
\quad\text{для  $u\in U$ с $\int \Phi\circ (u-v)\, d \mu\in (0,+\infty)$}.
\end{equation}
Если  в дополнение к \eqref{c:va} имеем также  $(v\circ \mathfrak a_x^y-v)\in U$, т.\,е.
\begin{equation}\label{c:va+}
(v-v\circ \mathfrak a_x^y)\in U\cap (-U)\neq \varnothing, \quad\text{где } -U:=\{-u\colon u\in U\},
\end{equation}
а  отображение ${{\mathsf u}}\mapsto {{\mathsf u}}\circ {\mathfrak a}_x^y$ --- сюръекция из $U\ni {\mathsf u}$ на $U$, то оценка \eqref{best} точна.
\end{theorem}
\begin{proof} {\rm Из равенств \eqref{0i0x}   Леммы в обозначениях \eqref{PhiQq} ввиду \eqref{c:va} получаем
\begin{equation}\label{equxy}
\frac{q\bigl(u(y)-v(y)+v(x)\bigr)}{Q\bigl(\int \Phi\circ (u-v)\, d \mu\bigr)}=\frac{q\bigl(u_x^y(x)\bigr)}{Q\bigl(\int \Phi\circ (u_x^y-v)\, d \mu\bigr)}\quad\text{для  функций $u\in U$,} \quad u_x^y {\in}  U,
\end{equation}
с конечными значениями интегралов $\int \Phi\circ (\cdot -v)\, d \mu\in (0,+\infty)$ в знаменателях дробей. 
Применяя к правой части \eqref{equxy} операцию $\sup$ по всем $u_x^y\in U$, в обозначениях  \eqref{PhiQq} имеем \eqref{best}. При условии \eqref{c:va+} и сюръективности отображения $u\mapsto u\circ {\mathfrak a}_x^y$ для любой функции ${{\mathtt u}}\in U$ найдется функция $u\in U$, для которой в обозначениях  \eqref{c:va} из  Леммы имеем ${{\mathtt u}}=u_x^y$. Следовательно, применение операции $\sup$ по всем $u_x^y\in U$ с ограничениями  $0<\|u_x^y\|<+\infty$ к правой части реализует  значение  $\|\delta_x\|^*$.  При этом  {\it  равенство\/} в \eqref{equxy}  для {\it любой\/} функции $u\in U$ с соответствующими предположениями о конечности $\|u\|$ доказывает точность  \eqref{best}.}
\end{proof}

\section{Голоморфная версия схемы}\label{sec:2}   
Пусть $X:={D}\subset \CC^n$ --- область,  $\mu\in \mathcal M^+({D})$ --- ненулевая мера, $p\in (0,+\infty)$ и для функции  $w\colon {D}\to  \RR$  существует интеграл  
\begin{equation}\label{intw}
 \int\limits_{{D}}e^{-pw}\d \mu < +\infty.
\end{equation} 
Для $f\in \Hol ({D})$ и ненулевого интеграла из \eqref{intw} введем квазинормы 
\begin{subequations}\label{nf}
\begin{align}
\|f\|_{p;w}&{:=}\left(\frac{1}{\int\limits_{{D}}e^{-pw}\d \mu}
\int\limits_{{D}}\bigl|f\bigr|^pe^{-pw}\d \mu\right)^{1/p}<+\infty ,
\tag{\ref{nf}a}\label{{nf}a}
\\
 \|f\|_{+\infty; w}&{:=} \sup_{z\in {D}} \Bigl(\bigl|f(z)\bigr|e^{-w(z)}\Bigr),
\tag{\ref{nf}b}\label{{nf}b}
\end{align}
\end{subequations} 
обобщающие соответственно \eqref{{ino}a} и \eqref{{ino}b}.

\begin{theorem}\label{th:2} Пусть   $0, {\mathsf z}\in {D}$  и 
голоморфное отображение $\mathfrak a=\mathfrak a_0^{\mathsf z}\colon {D}\to {D}$ удовлетворяет условию   {\rm (A)} с равенством  $\mathfrak a_0^{\mathsf z}(0){=}\mathsf z$ из \eqref{c:a}, мера $\mu$ инвариантна относительно  $\mathfrak a_0^{\mathsf z}$ и
\begin{equation}\label{expw}
\exp(w-w\circ \mathfrak a_0^{\mathsf z}){\in} \bigl|\Hol({D})\bigr|:=\bigl\{|\varphi|\colon \varphi \in \Hol ({D})\bigr\}.
\end{equation}
Тогда для $f\in \Hol ({D})$  при $\|f\|_{p;w}<+\infty$ в обозначениях \eqref{nf} имеет место оценка 
\begin{equation}\label{besth}
\bigl|f(\mathsf z)\bigr|e^{-w(\mathsf z)} {\leq} \|\delta_0\|_{p;w}^*
\cdot \|f\|_{p;w}\cdot e^{-w(0)},  \quad\text{где  }\|\delta_{0}\|_{p;w}^* :=\sup_{{\stackrel{0<\|\varphi \|_{p;w}<+\infty}{\varphi \in \Hol({D})}}} \frac{\bigl|\varphi (0)\bigr|}{\|\varphi \|_{p;w}}\,.
\end{equation}
Если   $\varphi \mapsto \varphi \circ {\mathfrak a}_0^{\mathsf z}$ --- сюръекция из  $\Hol (D)\ni \varphi$  на $\Hol ({D})$, то эта оценка  точна. Если такое отображение $\mathfrak a_0^{\mathsf z}$  существует для каждой точки $\mathsf z \in D$, то при всех $p\in (0,+\infty)$ оценка 
\begin{equation}\label{nfnp}
\|f\|_{+\infty; w} {\leq}  \|\delta_0\|_{p;w}^* \cdot \|f\|_{p; w}\cdot e^{-w(0)}\quad \text{неулучшаема}.
\end{equation}
\end{theorem}
\begin{proof} По Теореме {\rm \ref{th:1}} при $x:=0$ и $y:={\mathsf z}$ с классом $U:=\bigl\{\ln |f|\colon f \in \Hol ({D})\bigr\}$,  с функциями $\Phi (t){:=}e^{pt}$, $q(t){:=}e^t$ от $t\in \RR$, 
а также с функцией 
\begin{equation*}
Q(t){:=}t^{1/p}\,\biggl(\int_D  e^{-pw}\d \mu\biggr)^{-1/p}, \quad t\in (0,+\infty),
\end{equation*} 
в \eqref{Phi}  получаем оценку \eqref{besth}. 
Из \eqref{expw} следует представление   $w-w\circ \mathfrak a_0^{\mathsf z}=\ln |\psi|$
с $\psi \in \Hol (D)$ и, поскольку  $w(D)\subset \RR$, функция $\psi$ не имеет нулей в $D$ и 
$1/\psi \in \Hol (D)$. Таким образом, $w-w\circ \mathfrak a_0^{\mathsf z} \in -U$, выполнено условие    \eqref{c:va+}
и при условии сюръективности из Теоремы \ref{th:1} получаем точность оценки \eqref{besth}.
Когда точка $\mathsf z$ пробегает всю область  $D$,  из неравенства \eqref{besth} и из определения  \eqref{{nf}b} 
следует неулучшаемая оценка \eqref{nfnp}. 
\end{proof}
 \begin{remark} {\it Условие \eqref{expw} эквивалентно плюригармоничности функции $w-w\circ \mathfrak a_0^{\mathsf z}$ в области $D$, когда   область  $D$  звездная  с центром в нуле\/} \cite[Предложение 2.2.13]{Klimek}. Именно таковы области $D$ ниже в Примерах \ref{ex1}--\ref{ex3} и Замечаниях \ref{rem:2}, \ref{rem:3}, где при определении функций, отображений, а также мер через их плотность точка   $z=(z_1,\dots,z_n)\in \CC^n$ пробегает  назначенные области  $D\subset \CC^n$.
 \end{remark}

\section{Примеры применения голоморфной версии}\label{sec:3}

\begin{example}\label{ex1} Для  $D=\CC^n$ и для $\mathsf{z}=(\mathsf{z}_1, \dots, \mathsf{z}_n)\in \CC^n$ 
 подходящий выбор 
\begin{equation}\label{muCa}
\mu:=\lambda, \quad 
\mathfrak a_0^{\mathsf z}\colon z \mapsto z+{\mathsf z}\quad\text{---  {\it автоморфизм,\/}} \quad  
w({z}):=\frac{\alpha}{2}|{z}|^2
\end{equation} 
с {\it плюригармонической\/} по $z=(z_1, \dots, z_n)\in \CC^n$, $z_k\in \CC$,  функцией 
\begin{equation*}
w(z)-(w\circ \mathfrak a_0^{\mathsf z})(z)=-\frac{\alpha}{2}|\mathsf z|^2-\alpha \Re\,  \langle z,{\mathsf z}\rangle,
\end{equation*}  
где  
$\langle z,\mathsf{z}\rangle:= z_1\bar{\mathsf{z}}_1+\dots + z_n\bar{\mathsf{z}}_n$ --- скалярное произведение 
\cite[1.1.2(1)]{Rudin},
 $\|\delta_0\|_{p;w}^*{=}1$ ввиду  {\it субгармоничности\/} $|f|^p$
при $f\in \Hol({\CC^n})$ и {\it радиальности\/} $w$. Таким образом, Теорема \ref{th:2} сразу дает  Теорему F  для $F_{\alpha}^p(\CC^n)$  при $n\geq 1$.
\end{example}
\begin{example}\label{ex1+}
Другой возможный выбор при тех же $D=\CC^n$, $\mathsf z =({\mathsf z}_1,\dots,{\mathsf z}_n) \in \CC^n$, $\mu:=\lambda$ и   $\mathfrak a_0^{\mathsf z}$  из \eqref{muCa} --- это функция 
\begin{equation*}
w({z}):=\sum_{j=1}^n\frac{\alpha_j}{2}|z_j|^2\quad 
\text{в обозначении  $\vec\alpha= (\alpha_1,\dots,\alpha_n) \in (0,+\infty)^n$}
\end{equation*} 
  с {\it плюригармонической\/} по $z=(z_1,\dots, z_n)\in \CC^n$ функцией  
\begin{equation*}
w(z)-(w\circ \mathfrak a_0^{\mathsf z})(z)=-\sum_{j=1}^n\Bigl(\frac{\alpha_j}{2}|{\mathsf z}_j|^2+\alpha_j \Re\, z_j\bar{\mathsf z}_j\Bigr),
\end{equation*} 
где,  по-прежнему,  $\|\delta_0\|_{p;w}^*{=}1$ из  {\it плюрисубгармоничности\/} $|f|^p$ при $f\in \Hol({\CC^n})$ и {\it радиальности по каждой переменной\/} $z_j$ функции $w$. 

\begin{remark}\label{rem:2} Возможен и аналогичный выбор функции $w$, {\it радиально\/} зависящей от векторов-компонент из подпространств, представляющих $\CC^n$ в виде прямой суммы.   Теорема \ref{th:2} с точными оценками \eqref{besth}--\eqref{nfnp}  применима  к любым подобным конструкциям
с $\|\delta_0\|_{p;w}^*{=}1$ ввиду  {\it субгармоничности сужений\/} функции $f\in \Hol(\CC^n)$ на такие подпространства.
\end{remark}
\end{example}

\begin{example}\label{ex2} Для {\it единичного шара\/} $D:=\BB:=\{z\in \CC^n\colon |z|<1\}$ и точки ${\mathsf z}\in \BB$  выбираем меру $\mu$ через плотность \cite[2.2.6(ii)]{Rudin}
\begin{equation*}%%\label{key}
\d \mu (z):=\frac{1}{(1-|z|^2)^{n+1}}\d \lambda(z), \quad \mathfrak a_0^{\mathsf z} :=\varphi_{\mathsf z}
\quad\text{--- {\it автоморфизм шара $\BB$,\/}}
\end{equation*} 
 выписанный  в явном виде в \cite[2.2.1(2)]{Rudin} с требуемыми в Теореме \ref{th:2}  свойствами \cite[2.2.2(i), 2.2.2(vi), 2.2.6(ii)]{Rudin}. Полагаем также
\begin{equation*}%%\label{key}
w({z}):=-\frac{\alpha+n+1}{p}\ln \bigl(1-|{z}|^2\bigr), \quad \alpha \in (-1,+\infty), \quad p\in (0,+\infty).
\end{equation*} 
При этом  по тождеству \cite[2.2.2(iv)]{Rudin} получаем {\it плюригармоническую\/} по  $z=(z_1, \dots, z_n)\in \BB$ фу\-н\-к\-ц\-ию
\begin{equation*}%%\label{key}
w(z)-(w\circ \mathfrak a_0^{\mathsf z})(z)=
\frac{\alpha+n+1}{p}\Bigl (\ln \bigl|(1-\langle z,{\mathsf z}\rangle )^2\bigr|-\ln \bigl(1-|\mathsf z|^2\bigr) 
\Bigr).
\end{equation*} 
Вновь применима  Теорема \ref{th:2} с точными оценками \eqref{besth}--\eqref{nfnp}, где  $\|\delta_0\|_{p;w}^*{=}1$ так же, как в Примере \ref{ex1}.
Получаемые при таком выборе  подклассы  функций  $f\in \Hol (\BB)$ с $\|f\|_{p;w}<+\infty$ --- пространства
 Бергмана\,--\,Джрбашяна  $A_{\alpha}^p(\BB)$ \cite[1.3]{Schv}, \cite{Zh}, или  Фока\,--\,Баргмана  для шара $\BB$ в иной терминологии \cite[5.1.2]{Per}.
\end{example}

\begin{example}\label{ex3} Пусть $\DD:=\{z\in \CC\colon |z|<1\}$ --- единичный круг. Для  единичного поликруга 
$\DD^n\subset \CC^n$ и точки  $\mathsf z =({\mathsf z}_1,\dots,{\mathsf z}_n) \in \DD^n$ выбираем меру $\mu$ через ее плотность 
\begin{equation*}
\d \mu (z_1,\dots, z_n)=\bigotimes\limits_{j=1}^{n}\frac{1}{\bigl(1-|z_j|^2\bigr)^{2}}\d \lambda (z_j) , \quad z=(z_1,\dots, z_n)\in \DD^n,
\end{equation*} 
где в правой части первого равенства стоит тензорное произведение мер \cite[гл.~IV, \S~8]{Sch}, 
\begin{equation*}%%\label{key}
\mathfrak a_0^{\mathsf z}(z)=\Bigl(\frac{{\mathsf z_1}-z_1}{1-z_1\bar{\mathsf z}_1}, \dots, \frac{{\mathsf z_n}-z_n}{1-z_n\bar{\mathsf z}_n}\Bigr)\quad\text{---  {\it автоморфизм поликруга\/} \cite[7.3.3]{RudinP},} 
\end{equation*}
удовлетворяющий требованиям Теоремы \ref{th:2}. Полагаем также
\begin{equation*}%%\label{key}
w(z):=-\sum\limits_{j=1}^{n} \frac{\alpha_j +2}{p}\ln \bigl(1-|z_j|^2\bigr), \quad   \vec\alpha=
(\alpha_1,\dots,\alpha_n) \in (-1,+\infty)^n, \quad p\in (0,+\infty). 
\end{equation*}
В этом случае   {\it плюригармонична\/} по $z\in \DD^n$ функция 
\begin{equation*}%%\label{key}
w(z)-(w\circ \mathfrak a_0^{\mathsf z})(z)=
\sum_{j=1}^n\frac{\alpha_j+2}{p}\Bigl (\ln \bigl|(1- z_j\bar{\mathsf z}_j)^2\bigr|-\ln \bigl(1-|\mathsf z_j|^2\bigr) 
\Bigr). 
\end{equation*}  
Применима Теорема \ref{th:2} с точными оценками \eqref{besth}--\eqref{nfnp}, где  $\|\delta_0\|_{p,w}^*=1$ из  {\it плюрисубгармоничности\/} $|f|^p$
при $f\in \Hol({\DD^n})$ и {\it радиальности по каждой переменной $z_j$\/} функции $w$.
При этом  функции   $f\in \Hol (\DD^n)$ с $\|f\|_{p;w}<+\infty$ образуют   пространства
 Бергмана\,--\,Джрбашяна  $A_{\vec \alpha}^p(\DD^n)$ на  единичном поликруге  $\DD^n$ \cite[1.3]{Schv}. 

\begin{remark}\label{rem:3} Точные оценки \eqref{besth}--\eqref{nfnp} Теоремы \ref{th:2} можно легко получить и для пространств Бергмана\,--\,Джрбашяна голоморфных функций на произвольном поликруге/полицилиндре
\begin{equation*}
D=\prod_{j=1}^n r_j\DD,  \quad \vec r=(r_1,\dots, r_n)\in (0,+\infty]^n, \quad\text{где $(+\infty)\cdot \DD:=\CC$.}
\end{equation*}
Для этого следует применить  гомотетию по  переменным $z_j$ при $r_j<+\infty$, переводящую $r_j\DD$ в $\DD$, вместе с техникой Примера  \ref{ex1+} для тех $j$, для которых $r_j=+\infty$.
\end{remark}
\end{example}

\begin{remark} Теорема \ref{th:2} позволяет получать оценки, 
зачастую  точные, и других интегральных квазинорм в духе \cite[Теорема 2.10]{F}, \cite[Лемма 2.2]{Mas}. Для этого необходимо
последовательно применить  две операции: 
1) подействовать возрастающей функцией $F\colon [0,+\infty)\to \RR$ на обе части неравенства из \eqref{besth}; 2) проинтегрировать полученное неравенство по конечной мере ${\nu}\in \mathcal M^+(D)$. Тогда имеет место оценка
\begin{equation*}%%\label{besth}
\int_D F\left(\bigl|f({\mathsf z})\bigr|e^{-w(\mathsf z)}\right)\d \nu ({\mathsf z}){\leq} F\left(\|\delta_0\|_{p;w}^*\cdot \|f\|_{p;w}\cdot e^{-w(0)}\right)\nu(D)\,.
\end{equation*}
Достаточный произвол в выборе  $F$ и  $\nu$ обеспечивает разнообразие таких оценок.    
\end{remark}

Авторы глубоко признательны рецензенту за ряд полезных замечаний и исправлений.
%%\renewcommand{\refname}{\centering{\rm \scshape Литература}}

%% Информация об авторе: имя, отчество, фамилия, должность, кафедра, отдел, высшее учебное заведение (Академия %%наук, НИИ, ВЦ и т.д.), индекс служебного адреса, служебный адрес, e-mail.
\fullauthor{Баладай Рустам Алексеевич}
\address{450076, г. Уфа, ул. З. Валиди, 32, БашГУ, ФМиИТ, кафедра высшей алгебры и геометрии, аспирант}
\email{rbaladai@gmail.com}

% То же для второго автора
\fullauthor{Хабибуллин Булат Нурмиевич}
\address{450074, г. Уфа, ул. З. Валиди, 32, БашГУ, ФМиИТ, заведующий кафедрой высшей алгебры и геометрии, профессор}
\email{khabib-bulat@mail.ru}

% для третьего и т.д. 

%%%% Информация об авторе(ах) на английском языке

%%\fullauthor{ }
%%\address{    }
%%\email{ }
\fullauthor{Baladai Rustam Alekseevich}
\address{Postgraduate of the Chair of Higher Algebra and Geometry, Dept. of Math. \& IT, Bash. State Univ., Z. Validi Str., 32, Ufa, Bashkortostan, 450076, Russian Federation}
\email{rbaladai@gmail.com}

\fullauthor{Khabibullin Bulat Nurmievich}
\address{Prof., Head of the Chair of Higher Algebra and Geometry, Dept. of Math. \& IT, Bash. State Univ., Z. Validi Str., 32,  Ufa, Bashkortostan, 450076, Russian Federation}
\email{khabib-bulat@mail.ru}

%% .............

\label{lastpage}  %% Не удалять!!!


\begin{thebibliography}{86}
\bibitem{F} Zhu K.  {\it Analysis on Fock Spaces,\/} Graduate Texts in Mathematics, 263
(Springer-Verlag, 2012). 

\bibitem{Mas}Massaneda~X.,  Thomas~P.\,J. \textit{Interpolating sequences for Bargmann-Fock spaces in $\CC^n$,\/}
Indag. Mathem., N.S., {\bf 11} (1), 115--127  (2000). 

\bibitem{Lind}  Lindholm~N. \textit{Sampling in Weighted\/ $L^p$ Spaces of Entire Functions in\/
$\CC^n$ and Estimates of the Bergman Kernel,\/} Journal of Functional Analysis, {\bf 182}, 390--426  (2001).

\bibitem{BalKh} Баладай~Р.\,А., Хабибуллин~Б.\,Н. \textit{От интегральных оценок  функций к равномерным и локально усредненным,\/} Известия вузов. Математика (принято к печати 29 сентября 2016 г.).

\bibitem{Sch} Шварц~Л. \textit{Анализ, Т.~I\/} (Мир, М., 1972).

\bibitem{Klimek} Klimek~M.  \textit{Pluripotential Theory\/}   (Clarendon Press,  Oxford, 1991).

\bibitem{Rudin} Рудин~У. \textit{Теория функций в единичном шаре из\/ $\CC^n$\/} (Мир, М., 1984). 

\bibitem{Schv} С. В. Шведенко, \textit{Классы Харди и связанные с ними пространства аналитических
функций в единичном круге, поликруге и шаре,\/} Итоги науки и техн. Сер. Мат.
анал., ВИНИТИ, М., {\bf 23}, 3--124 (1985).

\bibitem{Zh} R. Zhao, K. Zhu, \textit{Theory of Bergman spaces in the unit ball of\/ $\CC^n$\/} 
(M\'emoires de la Soci\'et\'e Math\'ematique de France, {\bf 115},  2008). %%p. 103,

\bibitem{Per} Переломов А.М. \textit{Обобщенные когерентные состояния и их применения\/}
(Наука, М., 1987). %%, 272 стр.

\bibitem{RudinP} Рудин У. \textit{Теория функций в поликруге\/} (Мир, М., 1974).
\end{thebibliography}
\end{document}